\documentclass[12pt, reqno]{amsart}
\usepackage{amsmath, amsthm, amscd, amsfonts, amssymb, graphicx, color, enumerate}
\usepackage[bookmarksnumbered, colorlinks, plainpages]{hyperref}
\hypersetup{colorlinks=true,linkcolor=red, anchorcolor=green, citecolor=cyan, urlcolor=red, filecolor=magenta, pdftoolbar=true}

\newtheorem{theorem}{Theorem}[section]
\newtheorem{lemma}[theorem]{Lemma}
\newtheorem{proposition}[theorem]{Proposition}
\newtheorem{corollary}[theorem]{Corollary}
\theoremstyle{definition}

\theoremstyle{remark}

\numberwithin{equation}{section}

\begin{document}

\setcounter{page}{1}

\title[Generalized Module Extension Banach Algebras]{Generalized Module Extension Banach Algebras: Derivations and Weak Amenability}

\author[M. Ramezanpour$^{1}$, \MakeLowercase{and} S. Barootkoob$^2$]{Mohammad Ramezanpour,$^{1*}$ 
\MakeLowercase{and} Seddigheh Barootkoob$^2$}

\address{$^1$School of Mathematics and Computer Science,
Damghan University, P. O. Box 36716, Damghan 41167, Iran.}
\email{\textcolor[rgb]{0.00,0.00,0.84}{ramezanpour@du.ac.ir}}

\address{$^2$ Department of Mathematics, University of Bojnord, P.O. Box 1339, Bojnord, Iran.}
\email{\textcolor[rgb]{0.00,0.00,0.84}{s.barutkub@ub.ac.ir}}





\date{
\indent $^{*}$Corresponding author}

\begin{abstract}
Let $A$ and $X$ be Banach algebras and let $X$ be an algebraic Banach $A-$module. 
Then the $\ell^1-$direct sum $A\times X$ equipped with the multiplication
$$(a,x)(b,y)=(ab, ay+xb+xy)\quad (a,b\in A, x,y\in X)$$
is a Banach algebra, denoted by $A\bowtie X$, which will be called 
"{\it a generalized module extension Banach algebra}".  Module extension algebras, Lau product and also the 
direct sum of Banach algebras are the main examples satisfying this framework. We characterize the 
structure of $n-$dual valued ($n\in\mathbb N)$) derivations on $A\bowtie X$ from which we investigate the $n-$weak 
amenability for the algebra $A\bowtie X.$ 
We apply the results and the techniques of proofs for presenting some older results with simple direct proofs.
\end{abstract} \maketitle

\section{Introduction and some Preliminaries}
A derivation from a Banach algebra $A$ to a
Banach $A$-module $X$ is a bounded linear mapping $D:A\rightarrow X$
such that $D(ab)=D(a)b+aD(b)$ for all $a,b\in A$. For each $x\in
X$ the mapping $d_x:a\rightarrow ax-xa$, $(a \in A)$ is a
derivation, called the inner derivation implemented by $x$. The concept of $n-$weak amenability 
was introduced and intensively studied by Dales, Ghahramani and Gr{\o}nb{\ae}k \cite{DGG}. A Banach 
algebra $A$ is said to be $n-$weakly amenable ($n\in\mathbb{N}$) if every
derivation from ${\mathcal A}$ into $ A^{(n)}$ is
inner, where $A^{(n)}$ is the $n^{th}-$dual of $A$, we also write $A^{(0)}=A$. $1-$weak amenability is the so-called weak
amenability which was first introduced and studied by Bade, Curtis
and Dales \cite{BCD} for commutative Banach algebras and then by
Johnson \cite{J} for a general Banach algebra. $C^*-$ algebras and the convolution group 
algebras are the main examples of $n-$weakly amenable Banach algebras (for a proof see \cite{DGG}, \cite{Z1}). 
For more background concerning $n-$weak amenability one may refer  to the monograph \cite{D}.



Let $A$ and $X$ be Banach algebra and let $X$ be a Banach $A-$module.
We say that $X$ is an algebraic Banach $A-$module if for every $x, y\in X$ and $a\in A,$
\begin{align*}
&a(xy)=(ax)y,\ 
(xy) a=x(ya),\ 
(xa)y=x(ay),\ \ \  \rm{and}\\ 
&\|ax\|\leq\|a\|\|x\|,\ \  \|xa\|\leq\|a\|\|x\| 
\end{align*}
\[.\]
Then a direct verification shows that the $\ell^1-$direct sum $A\bowtie X$ under the multiplication
\[(a,x)(b,y)=(ab, ay+xb+xy)\quad (a,b\in A, x,y\in X)\]
is a Banach algebra which is called the generalized module extension of $A$ by $X$ and denoted by $A\bowtie X$.

It is easy to check that $A\times\{0\}$ is a closed subalgebra while $\{0\}\times X$ is a
closed ideal of $A\bowtie X$, and that $(A\bowtie X)/(\{0\}\times X)$ is isometrically isomorphic to $A\times\{0\}$.\\

The main examples of generalized module extensions are listed as follows:
\begin{enumerate}[\hspace{1em}\rm $\bullet$]
\item (\textit{The classical module extension algebras}) The (classical) module extension algebra $A\ltimes X$, as introduced in \cite{BDL}, 
is the $\ell^1-$direct sum of $A$ by the Banach $A-$module $X$ equipped with the multiplication
\[(a,x)(b,y)=(ab, ay+xb)\quad (a,b\in A, x,y\in X).\]
Clearly, $A\ltimes X$ is a generalized module extension Banach algebra, in which $X$ is equipped with the trivial product $xy=0.$ 
Module extension Banach algebras are known as a rich source of (counter-)examples in various situations in abstract harmonic analysis and
functional analysis. Some aspects of the Banach algebra $A\ltimes X$ have been 
described in \cite{BDL}. In \cite{Z} Zhang has characterized the structures 
of derivations from $A\ltimes X$ into various duals from which he investigated 
the $n-$weak amenability properties of $A\ltimes X$.
The class of module extension Banach algebras also includes
the triangular Banach algebra whose weak amenability has been investigated in \cite{FM}.

\item (\textit{$\theta-$Lau product of Banach algebras}) Let $A$ and $B$ be Banach algebras 
and let $\theta\in\sigma(A).$ Then the $\theta-$Lau product 
$A{~}_{\theta\!\!}\times B$ which is equipped with the multiplication
\[(a,b)(c,d)=(ac, \theta(a)d+\theta(c)b+bd)\quad (a,c\in A, b,d\in B),\]
can be viewed as a generalized module extension algebra, in which $B$ endowed
with the module operations $ab=ba=\theta(a)b$ is an algebraic Banach $A-$module. 
This product was introduced by Lau \cite{L} for certain class of Banach algebras 
and followed by Sangani Monfared \cite{M} for the general case. A very 
familiar example is the case that $A=\mathbb{C}$ with $\theta=\iota$ as the identity 
character, that we get the unitization $B^\sharp=\mathbb{C}{~}_{\iota\!\!}\times B$ of $B$.
Some aspects of $A{~}_{\theta\!\!}\times B$ are investigated in \cite{EK, EK1, K}. 
In particular, the $n-$weak amenability of $A{~}_{\theta\!\!}\times B$ are discussed in \cite{EK1}.

\item (\textit{$T-$Lau product of Banach algebras}) Let $A$ and $B$ be Banach algebras 
and let $T:A\to B$ be a continuous homomorphism with $\|T\|\leq 1$. Then the $T-$Lau product
$A{~}_{T}\!\!\times B$ which is equipped with the multiplication
\[(a,b)(c,d)=(ac, T(a)d+bT(c)+bd)\quad (a,c\in A, b,d\in B),\]
can also be viewed as a generalized module extension algebra, in which $B$ furnished 
with the module operations $ab=T(a)b$ and $ba=bT(a)$ is an algebraic Banach $A-$module. 
This type of product was first introduced by
Bhatt and Dabhi \cite{Bh-D} for the case where $B$ is commutative and was extended by
Javanshiri and Nemati for the general case \cite{Javan-N}; see also \cite{NJav}.

In particular, in the case $T=0$, we arrive at the $\ell^1-$direct sum $A\oplus_1 B$ of two Banach algebra $A$ and $B$, equipped with the pointwise multiplication
\[(a,b)(c,d)=(ab,cd)\quad (a,c\in A, b,d\in B).\]
Recently, Choi \cite{Choi} has demonstrated that the $T-$Lau product $A{~}_{T}\!\!\times B$ is isomorphic as a Banach algebra to the usual direct
sum $A\oplus_1 B$.
\end{enumerate}
For a generalized module extension Banach algebra $A\bowtie X$ one can 
directly checked that the $n^{th}-$dual $(A\bowtie X)^{n}$ as a 
Banach $(A\bowtie X)-$module enjoys the following module operations:
\begin{align*}
(F,G) (a,x)&=(Fa,Fx
+G a+Gx),\\
(a,x) (F,G)&=(aF,xF
+aG+xG),\\
(f,g) (a,x)&=(fa+g x
,g x+ga),\\
(a,x)(f,g)&=(af+x g,
x g+ag),
\end{align*}
for all $a\in A, F\in A^{(2n)}, f\in A^{(2n+1)},
x\in X, G\in X^{(2n)}$ and $ g\in X^{(2n+1)}.$

The paper is organized as follows: In Section 2 we investigate $n-$weak amenability of 
$A\bowtie X$ in the case where $n$ is odd. Then we apply our results for the particular cases, 
$A\ltimes X,$ $A{~}_{\theta\!\!}\times B,$ $A{~}_{T}\!\!\times B$ and $A\oplus_1 B.$ 
Section 3 follows the same discipline for the case $n$ is even. 
The case where $X$ is unital is invistigated in Section 4. This special case provides some 
simplifications in the constructions of derivations to various duals of $A\bowtie X$.
\section{$(2n+1)-$weak amenability of $A\bowtie X$}
In this section we characterize $(2n+1)-$weak amenability of the 
generalized module extension Banach algebra $A\bowtie X$ in terms of 
the $(2n+1)-$weak amenability of $A$ and $X$. We begin with the following elementary 
lemma characterizing derivations from $A\bowtie X$ into $(A\bowtie X)^{(2n+1)}$.
\begin{lemma}\label{derivodd} 
Every derivation
$D: A\bowtie X\to (A\bowtie X)^{(2n+1)}\quad (n\in\mathbb N\cup\{0\})$ enjoys the presentation
\begin{equation}\label{der1}
{D}(a,x)=(D_{A}(a)+T_{A}(x),D_{X}(a)+T_{X}(x)),
\end{equation}
where
\begin{enumerate}[\hspace{1em}\rm (a)]
\item $D_{A}:A\to A^{(2n+1)}$ and $D_{X}:A\to X^{(2n+1)}$ are derivations.
\item $T_{A}:X\to A^{(2n+1)}$ is a bounded linear map such that for every $a\in A, x\in X,$\\
$T_{A}(a x)=a T_{A}(x)+D_{X}(a) x$ and
$T_{A}(x a)=T_{A}(x) a+x D_{X}(a).$ 
\item $T_{X}: X\to X^{(2n+1)}$ is a derivation such that for every $a\in A, x\in X,$\\
$T_X(a x)=D_X(a) x+a T_X(x)$,
$T_X(x a)=x D_X(a)+T_X(x) a$ and
$T_X(x) y+x T_X(y)=T_A(xy).$ 
\end{enumerate}
Moreover, $D$ is inner, that is, $D=d_{(f,g)}$ for some $f\in A^{(2n+1)},$ $g\in X^{(2n+1)}$
if and only if $D_{A}=d_{f}, D_{X}=d_{g}$ and
$T_{X}=d_{g}$ are inner derivations and $T_{A}=\delta_{g}$, where
$\delta_{g}(x)=xg-gx$, $(x\in X)$.
\end{lemma}

We are ready to prove one of our main result characterizing $(2n+1)-$amenability of $A\bowtie X.$
\begin{theorem}\label{wamenodd}
A generalized module extension Banach algebra $A\bowtie X$ is $(2n+1)-$weakly amenable if and only if
\begin{enumerate}[\hspace{1em}\rm (1)]
\item $A$ is $(2n+1)-$weakly amenable.
\item If $T_X:X\to X^{(2n+1)}$ is a derivation such that there exist a derivation
$D_X:A\to X^{(2n+1)}$ and a bounded linear map $T_A:X\to A^{(2n+1)}$ satisfying
$T_A(a x)=a T_A(x)+D_X(a)x,$
$T_A(x a)=T_A(x) a+xD_X(a),$
$T_X(a x)=D_X(a) x+a T_X(x),$
$T_X(x a)=x D_X(a)+T_X(x) a$ and
$T_X(x) y+x T_X(y)=T_A(xy),$ for all $a\in A, x\in X,$ then
$T_X$ is inner.
\item If $D_X:A\to X^{(2n+1)}$ is a derivation such that
$x D_X(a)=D_X(a) x=0$,  for all $a\in A, x\in X,$ then
there exists an element $g\in X^{(2n+1)}$ such that $D_X=d_{g}$ and $x g=g x$,  for all $x\in X$.
\item If $T_A:X\to A^{(2n+1)}$ is a bounded $A-$module homomorphism such that
$T_A(xy)=0$ for all $x,y\in X,$ then $T_A=0$.
\end{enumerate}
\end{theorem}
As a consequence we have the following result which extends \cite[Theorem 2.4]{EK1}.
\begin{proposition}\label{nweak-odd}
Suppose that both 
$A$ and $X$ are $(2n+1)-$weakly amenable. 
Then $A\bowtie X$ is $(2n+1)-$weakly amenable when either of the following condition holds.
\begin{enumerate}[\hspace{1em}\rm (1)]
\item $\langle XX^{(2n)}+X^{(2n)}X\rangle$ is dense in $X^{(2n)}$.
\item $\langle XA^{(2n)}+A^{(2n)}X\rangle$ is dense in $X^{(2n)}$.
\end{enumerate}
\end{proposition}
%

Applying Theorem \ref{wamenodd} for the calassical module extension Banach 
algebra $A\ltimes X$ we arrive at the following result which has already proved by Zhang \cite{Z} by a slightly different method.
\begin{corollary}[{\cite[Theorem 2.1]{Z}}]
A (classical) module extension Banach algebra $A\ltimes X$ is $(2n+1)-$weakly amenable if and only if
\begin{enumerate}[\hspace{1em}\rm (1)]
\item $A$ is $(2n+1)-$weakly amenable.
\item The only $A-$module morphism $T:X\rightarrow X^{(2n+1)}$ such that $xT(y)+T(x)y=0$, in $A^{(2n+1)}$, for all $x,y\in X$, is zero.
\item $H^1(A,X^{(2n+1)})=\{0\}$; that is, every derivation from $A$ to $X^{(2n+1)}$ is inner.
\item For every continuous $A-$module morphism $S:X\rightarrow A^{(2n+1)}$, 
there exists $f\in X^{(2n+1)}$ such that $af=fa$ for all $a\in A$ and $S(x)=xf-fx$ for all $x\in X$.
\end{enumerate}
\end{corollary}
\begin{proof}
It is enough to use Theorem \ref{wamenodd} for the case where $xy=0$ for all $x,y\in X.$
\end{proof}

As another consequence of  Theorem \ref{wamenodd}, we use it for the $\theta-$Lau product 
Banach algebra $A{~}_{\theta\!\!}\times B.$ Then we 
 get the following characterization for $(2n+1)-$weak amenability of $A{~}_{\theta\!\!}\times B$ 
which extends the related results in \cite{EK1}.
\begin{corollary}\label{Lauodd}
The $\theta-$Lau product Banach algebra $A{~}_{\theta\!\!}\times B$ is $(2n+1)-$weakly amenable if and only if
\begin{enumerate}[\hspace{1em}\rm (1)]
\item $A$ is $(2n+1)-$weakly amenable.
\item Every derivation $T:B\rightarrow B^{(2n+1)}$ for which there exists a bounded 
$A-$module homomorphism $S:B\rightarrow A^{(2n+1)}$ such that $\left(T(b)(d)+T(d)(d)\right)\theta=S(bd)$ for all $b,d\in B$ is inner.
\item The only bounded linear map $D:A\to B^{(2n+1)}$ such that
$D(ac)=\theta(a)D(c)+\theta(c)D(a)$ and 
$bD(a)=0=D(a)b$ for all $a,c\in A$ and $b\in B$, is zero.
\item 
The only bounded linear operator $S:B\to A^{(2n+1)}$ such that 
$S(bd)=0, aS(b)=S(b)a=\theta(a)S(b)$ for all $a\in A$ and $b,d\in B$, 
is zero.
\end{enumerate}
\end{corollary}

Applying Theorem \ref{wamenodd} for the case where $X=A$, with the multiplication as the module operation, we 
have the following characterization of $(2n+1)-$weak 
amenability of $A\bowtie A$.
\begin{corollary}\label{B=Aodd}
$A\bowtie A$ is $(2n+1)-$weakly amenable if and only if
$A$ is $(2n+1)-$weakly amenable.
\end{corollary}


Applying Theorem \ref{wamenodd} for the direct sum algebra $A\oplus_1 B$ 
we get the following well known result.
\begin{corollary}\label{zeromododd}
$A\oplus_1 B$ is $(2n+1)-$weakly amenable if and only if
both $A$ and $B$ are $(2n+1)-$weakly amenable.
\end{corollary}

\section{$(2n)-$weak amenability of $A\bowtie X$}
In this section we characterize $(2n)-$weak amenability of the generalized module 
extension Banach algebra $A\bowtie X$ in terms of the $(2n)-$weak amenability 
of $A$ and $X$. Similar to the previous section (Lemma \ref{derivodd}) we begin with the following elementary 
lemma characterizing the derivations from $A\bowtie X$ into $(A\bowtie X)^{(2n)}$.
\begin{lemma}\label{deriveven} Every derivation
$D: A\bowtie X\to (A\bowtie X)^{(2n)},\quad (n\in\mathbb N\cup\{0\})$, enjoys the presentation
\begin{equation}\label{der2}
{D}(a,x)=(D_{A}(a)+T_{A}(x),D_{X}(a)+T_{X}(x)),
\end{equation}
where
\begin{enumerate}[\hspace{1em}\rm (a)]
\item $D_{A}:A\to A^{(2n)}$ and $D_{X}:A\to X^{(2n)}$ are derivation.
\item $T_{A}:X\to A^{(2n)}$ is a bounded $A$-module homomorphism satisfying
$T_{A}(xy)=0$, for all $x,y\in X$.
\item $T_{X}:X\to X^{(2n)}$ is a bounded linear map such that
$T_X(a x)=D_A(a) x+D_X(a) x+a T_X(x),$ 
$T_X(x a)=x D_A(a)+x D_X(a)+T_X(x) a$ and
$T_X(xy)=x T_X(y)+T_X(x) y+x T_A(y)+T_A(x) y,$ for each $a\in A, x,y\in X.$
\end{enumerate}
Moreover, $D=d_{(F,G)}$ for some $F\in A^{(2n)}, G\in X^{2n},$ if and only if
$D_{A}=d_{F}, D_{X}=d_{G}$ and
$T_{X}=d_{G}+\delta_{F}$ (are inner derivations), and $T_{A}=0$, where
$\delta_{F}(x)=x F-F x,\ (x\in X)$.
\end{lemma}
The next result is the $(2n)-$version of Theorem \ref{wamenodd}.
\begin{theorem}\label{wameneven}
A generalized module extension Banach algebra $A\bowtie X$ is $(2n)-$weakly amenable if and only if
\begin{enumerate}[\hspace{1em}\rm (1)]
\item If $D_A:A\to A^{(2n)}$ is a derivation such that there exist a 
derivation $D_X:A\to X^{(2n)}$, a bounded linear operator $T_X:X\to X^{(2n)}$ and a bounded $A$-module
homomorphism $T_A:X\to A^{(2n)}$ satisfying $T_A(xy)=0$,
$T_X(a x)=D_A(a) x+D_X(a) x+a T_X(x),$
$T_X(x a)=x D_A(a)+x D_X(a)+T_X(x) a$ and
$T_X(xy)=x T_X(y)+T_X(x) y+T_A(x) y+x T_A(y)$, 
for all $a\in A$ and $x,y\in X$,
then $D_A$ is inner.
\item If $T_X:X\to X^{(2n)}$ is a bounded linear operator such that there exist a derivation
$D_X:A\to X^{(2n)}$ and a bounded $A$-module
homomorphism $T_A:X\to A^{(2n)}$ satisfying $T_A(xy)=0$,
$T_X(a x)=D_X(a) x+a T_X(x),$
$T_X(x a)=x D_X(a)+T_X(x) a$ and
$T_X(xy)=x T_X(y)+T_X(x) y+T_A(x) y+x T_A(y)$, 
for all $a\in A$ and $x,y\in X$, 
then there exist elements $F\in A^{(2n)}$ and $G\in X^{(2n)}$ such that 
$T_X=\delta_{F}+d_{G}$ and $a F=F a$ for all $a\in A$.
\item If $D_X:A\to X^{(2n)}$ is a derivation such that
$x D_X(a)=D_X(a) x=0$, in $X^{(2n)}$, for all $a\in A$ and $x\in X$, 
then there exist elements $G\in X^{(2n)}$ and $F\in A^{(2n)}$ such that
$D_X=d_{G},~ d_{G}+\delta_{F}=0$ and $a F=F a$ for all $a\in A$.
\item If $T_A:X\to A^{(2n)}$ is a bounded $A$-module homomorphism such that 
$T_A(xy)=0$ for all $x,y\in X$, and there exists a bounded $A$-module homomorphism $T_X:X\to X^{(2n)}$ satisfying
$T_X(xy)=x T_X(y)+T_X(x) y+T_A(x) y+x T_A(y)$, for all $x,y\in X$, then $T_A=0$.
\end{enumerate}
\end{theorem}

As an application of the latter Theorem we bring the following even analogue of Proposition \ref{nweak-odd}.
\begin{proposition}\label{nweak-even}
Suppose that $X^2$ is dense in $X$ and both 
$A$ and $X$ are \break
$(2n)-$weakly amenable. 
If $\langle XX^{(2n-1)}+X^{(2n-1)}X\rangle$ is dense in $X^{(2n-1)}$,
then $A\bowtie X$ is $(2n)-$weakly amenable.
\end{proposition}
%

We apply Theorem \ref{wameneven} for a classical generalized module extension Banach algebra 
$A\ltimes X$ to obtain the following result of Zhang \cite{Z}.
\begin{corollary}[{\cite[Theorem 2.2]{Z}}]
A (classical) module extension Banach algebra $A\ltimes X$ is $(2n)-$weakly amenable if and only if the following conditions hold.
\begin{enumerate}[\hspace{1em}\rm (1)]
\item The only derivation $D:A \rightarrow A^{(2n)}$ for which there exists a 
continuous operator $T:X \rightarrow X^{(2n)}$ such that 
$T(ax)=aT(x)+D(a)x$ and $T(xa)=T(x)a+xD(a)$ for all $a\in A$ and $x\in X$, are inner derivations.
\item For every continuous $A-$module morphism $T:X\rightarrow X^{(2n)}$, there 
exists a $F\in A^{(2n)}$ such that $T(x)=xF-Fx$ and $aF=Fa$ for all $x\in X, a\in A$.
\item $H^1(A,X^{(2n)})=\{0\}$.
\item The only continuous $A-$module morphism $S:X\rightarrow A^{(2n)}$ 
for which $xS(y)+S(x)y=0$, in $X^{(2n)}$, for all $x,y\in X,$ is zero.
\end{enumerate}
\end{corollary}
\begin{proof}
It is enough to use Theorem \ref{wameneven} for the case where $xy=0$ for all $x,y\in X.$
\end{proof}

Applying Theorem \ref{wameneven} for the $\theta-$Lau product Banach algebra $A{~}_{\theta\!\!}\times B,$ 
we get the following characterization for $(2n)-$weak amenability of $A{~}_{\theta\!\!}\times B$ 
which extends the related results in \cite{EK1}. 
\begin{corollary}\label{Laueven}
A $\theta$-Lau product Banach algebra $A{~}_{\theta\!\!}\times B$ is $(2n)-$weakly amenable if and only if the following conitions hold.
\begin{enumerate}[\hspace{1em}\rm (1)]
\item The only derivations $D:A\to A^{(2n)}$ for which there is a
derivation $D_1:A\to B^{(2n)}$ such that
$D_1(a) b=b D_1(a)= -D(a)(\theta)b$ for all $a\in A$ and $b\in B,$
are inner derivations.
\item $B$ is $(2n)-$weakly amenable.
\item The only bounded linear operator $D_1:A\to B^{(2n)}$ such that
$D(ac)=\theta(a)D(c)+\theta(c)D(a)$ and
$bD_1(a)=0=D_1(a)b$ for all $a,c\in A, b\in B,$
is zero.
\item The only bounded linear operator $S:B\to A^{(2n)}$ for which
$S(bd)=0, aS(b)=S(b)a=\theta(a)S(b)$ for all $a\in A, b,d\in B$ and there is a 
bounded linear operator $T:B\to B^{(2n)}$ such that
$T(bd)=b T(d)+T(b) d+S(b)(\theta) d+S(d)(\theta) b$ for all $b,d\in B,$ is zero.
\end{enumerate}
\end{corollary}

Applying Theorem \ref{wameneven} for the case $X=A$, with the multiplication as module operation, we 
get the following characterization of $(2n)-$weak 
amenability of $A\bowtie A$.
\begin{corollary}\label{B=Aeven}
$A\bowtie A$ is $(2n)-$weakly amenable if and only if
\begin{enumerate}[\hspace{1em}\rm (1)]
\item $A$ is $(2n)-$weakly amenable.
\item the only bounded homomorphism $S:A\to A^{(2n)}$ for which
$aS(c)=S(a)c=S(ac)=0$ for each $a, c\in A$, is zero.
\end{enumerate}
\end{corollary}


Applying Theorem \ref{wameneven} for the direct sum algebra $A\oplus_1 B$ 
we get the following result.
\begin{corollary}\label{zeromodeven}
The direct sum $A\oplus_1 B$ of two Banach algebra $A$ and $B$ is $(2n)-$weakly amenable if and only if
\begin{enumerate}[\hspace{1em}\rm (1)]
\item Both $A$ and $B$ are $(2n)-$weakly amenable.
\item The only bounded homomorphism $D:A\to B^{(2n)}$ for which
$D(ac)=0$ and $b D(a)=D(a) b=0$ for all $a,c\in A, b\in B,$ is zero.
\item The only bounded homomorphism $S:B\to A^{(2n)}$ for which
$S(bd)=0$ and $a S(b)=D(b) a=0$ for all $a\in A, b,d\in B,$ is zero.
\end{enumerate}
\end{corollary}

If we combine Propositions \ref{nweak-odd} and \ref{nweak-even},
we have the following result providing some sufficient conditions 
for $n-$weak amenability of $A\bowtie X$. In particular, in the setting of 
 $\theta-$Lau products, it improves \cite[Proposition 2.4]{EK1}, as well as, 
for $T-$Lau products, it improve \cite[Proposition 3.5]{Javan-N}.
\begin{proposition}\label{nweak}
Suppose that $X^2$ is dense in $X$ and for some $n>0$ either
$\overline{X X^{(n-1)}}=X^{(n-1)}$ or $\overline{X^{(n-1)} X}=X^{(n-1)}$.
If $A$ and $X$ are $n-$weakly amenable
then $A\bowtie X$ is $n-$weakly amenable.
\end{proposition}

We recall from \cite[Proposition 1.2]{DGG} that if $A$ is weakly amenable, then $\overline{AA}=A$. 
Thus
as a rapid consequence of Proposition \ref{nweak} we get,
\begin{corollary}
If $A$ and $X$ are weakly amenable then $A\bowtie X$ is weakly amenable.
\end{corollary}

From Corollaries \ref{zeromododd}, \ref{zeromodeven} and Proposition \ref{nweak} we get the following result.
\begin{corollary}
Suppose that $B^2$ is dense in $B$ and for some $n>0$ either
$\overline{B B^{(n-1)}}=B^{(n-1)}$ or $\overline{B^{(n-1)} B}=B^{(n-1)}$. Then
$A\oplus_1 B$ is $n-$weakly amenable if and only if
both $A$ and $B$ are $n-$weakly amenable.
\end{corollary}

From Corollaries \ref{B=Aodd}, \ref{B=Aeven} we immediately obtain the next result.
\begin{corollary}
Suppose that $A^2$ is dense in $A$ then $A\bowtie A$ is $n-$weakly amenable if and only if
$A$ is $n-$weakly amenable.
\end{corollary}
\section{The case where $X$ is unital}
In this section we assume that $X$ is unital with the identity $1_X$.
In this case the characterizations of derivations $D:A\bowtie X\to(A\bowtie X)^{(n)}$ presented in Lemma
\ref{derivodd} and \ref{deriveven} can be considerably simplified. This result extends
\cite[Corollareis 3.9, 3.10]{Javan-N} and \cite[Proposition 3.1]{EK1}.
\begin{lemma}
Let $X$ be unital with the identity $1_X$. Then
\begin{enumerate}[\hspace{1em}\rm (1)]
\item $D:A\bowtie X\to (A\bowtie X)^{(2n+1)}$ is a derivation if and only if
$$D(a,x)=(D_{A}(a)+T_{X}(x) 1_X,T_{X}(a 1_X)+T_{X}(x)),\qquad(a\in A, x\in X);$$
where,
$D_{A}:A\to A^{(2n+1)}$ and $T_{X}:X\to X^{(2n+1)}$ are derivations and
$T_{X}(1_X a)=T_{X}(a 1_X)$ for all $a\in A$ and 
$T_{X}(x) 1_X=1_X T_{X}(x)$, in $A^{(2n+1)}$, for all $x\in X$.
Moreover, $D=d_{(f,g)}$ is inner derivation
if and only if $D_{A}=d_{f}$ and
$T_{X}=d_{g}$ are inner derivations.
\item $D:A\bowtie X\to (A\bowtie X)^{(2n)}$ is a derivation if and only if
$$D(a,x)=(D_{A}(a),T_{X}(a 1_X)-D_A(a) 1_X+T_{X}(x)),\qquad(a\in A, x\in X);$$
where,
$D_{A}:A\to A^{(2n)}$ and $T_{X}:X\to X^{(2n)}$ are derivations,
$T_{X}(1_X a)-1_X D_A(a)=T_{X}(a 1_X)-D_A(a) 1_X$ for all $a\in A$
and $D(a)=D_A(a) 1_X$ is a bounded derivation from $A$ into
$X^{(2n)}$.
Moreover, $D=d_{(F,G)}$ is inner derivation if and only if
$D_{A}=d_{F}$ and
$T_{X}=d_{G}+\delta_{F}=d_{G+1_X F}$ are inner derivations.
\end{enumerate}
\end{lemma}

As a consequence of 
Proposition \ref{nweak} we give the next result
concerning to the $(n)-$weak amenability of $A\bowtie X$
in the case where $X$ is unital. This result covers \cite[Theorem 3.1]{EK1} and 
\cite[Proposition 3.11]{Javan-N} in the special case.

\begin{theorem}
Suppose that $X$ is unital.
\begin{enumerate}[\hspace{1em}\rm (1)]
\item If $A$ and $X$ are $(2n+1)-$weakly amenable then
$A\bowtie X$ is $(2n+1)-$weakly amenable.
\item If $A\bowtie X$ is $(2n+1)-$weakly amenable then $A$ is $(2n+1)-$weakly amenable
and the only derivations $T_{X}:X\to X^{(2n+1)}$ which is also an $A-$module
homomorphism and $T_{X}(x) 1_X=1_X T_{X}(x)$ are inner derivations.
\item If $A\bowtie X$ is $(2n+1)-$weakly amenable then $A$ is $(2n+1)-$weakly amenable
and $X$ is $(2n+1)-$cyclicly weak $A$-module amenable.
\item If $A$ and $X$ are $(2n)-$weakly amenable then
$A\bowtie X$ is $(2n)-$weakly amenable.\\

\end{enumerate}
\end{theorem}

{\bf Acknowledgments.} The authors would like to express their appreciation
to Professor H. R. Ebrahimi Vishki for  introducing them to the subject of 
this work and all of his encouragement and valuable comments 
which provided significant improvements to this article.


\end{document}